\documentclass[12pt]{article}%
\usepackage{amsfonts}
\usepackage{amsmath}
\usepackage{amssymb}
\usepackage{mathptmx}
\usepackage{graphicx}
\usepackage{url}
\usepackage{subfigure}
\usepackage{color}
\usepackage{setspace}%
\providecommand{\U}[1]{\protect\rule{.1in}{.1in}}
\newtheorem{theorem}{Theorem}

\newtheorem{definition}[theorem]{Definition}

\newenvironment{proof}[1][Proof]{\textbf{#1.} }{\ \rule{0.5em}{0.5em}}

\newtheorem{alphatheorem}{Theorem}

\begin{document}

\title{Simultaneous elections in a polarized society make single-party sweeps more likely}
\author{Pradeep Dubey\thanks{Center for Game Theory, Economics Department, Stony Brook
University, Stony Brook, NY} \thinspace and Siddhartha Sahi\thanks{Department
of Mathematics, Rutgers University, New Brunswick, NJ} \thanks{The research of
S.S. was supported in part by Simons Foundation grant 00006698.}}
\date{}
\maketitle

\section*{Abstract}

In a country with many elections, it may prove economically expedient to
hold multiple elections simultaneously on a common polling date. We show that 
in a polarized society, in which each voter has a preferred party, an increase in 
the simultaneity of polling will increase the likelihood of a single-party
sweep, namely, it will become more likely that a single party wins all the elections. 

In fact we show that the sweep probability goes up for \emph{every} party. Thus the 
phenomenon we describe is independent of the ``coattail'' or ``down-ballot'' 
effect of a popular leader. It is a \emph{systemic} and \emph{persistent} 
macroscopic political change, effected by a combination of political polarization and 
simultaneity of polling.
   
Our result holds under fairly general conditions and is
applicable to many common real-world electoral systems, including
\emph{first-past-the-post} (most voters) and \emph{party list proportional
representation} (most countries). In the course of our proof, we obtain a
generalization of the well-known Harris correlation inequality.

\textbf{MSC Classification: }60C05 60E15 91B12. \quad\textbf{JEL
Classification: } D72

\textbf{Keywords: } Simultaneous elections, single-party sweeps, Harris inequality.

\section{Introduction\label{simult elections}}


Most major democracies around the world conduct multiple elections, allowing
citizens to engage in the democratic process at regular intervals across
various levels of government. Americans participate in elections every two
years to choose a member of the House of Representatives, every four years to
elect a President, and twice within a six-year period to vote for two
Senators; additionally, various local and municipal elections occur throughout
these cycles. In India, the national parliament and state assemblies are
legally mandated to have five-year terms; however, as in many parliamentary
democracies, the fall of a government can lead to elections being held earlier
than the scheduled term.

Some elections are tradionally held together with others, so that a voter on a
given poll date might cast votes in multiple elections. For example, the three
major American election cycles are synchronized so that every other House
election coincides with a Presidential election, and, for each voter, two out
of three House elections also involve a senatorial candidate on the ballot. In
India, the national and state elections were held simultaneously until 1967
but the premature fall of various governments have caused them to become
unsynchronized over the years.

Simultaneous elections reduce cost, time, and effort, and consequently there
is a strong economic argument to be made in their favor. The ruling party in
India, has proposed a constitutional amendent---``One Nation One
Election''---which seeks to mandate the simultaneity of national and state
elections. This has been opposed by other parties out of concern that the
current Prime Minister's popularity might increase the state-level dominance
of the ruling party. Indeed, the tendency of a popular leader to attract votes
for other candidates of the same party is a well-known and much-studied
phenomenon in the American context, where it is referred to as a ``coattail''
or ``down-ballot'' effect.

A separate development, observed and studied in both the American and Indian
contexts, is that of increasing \textquotedblleft political
polarization\textquotedblright\footnote{See, e.g., \cite{Bertoa},
\cite{Hetzel}, \cite{Sahoo} and the references therein.}. This refers to a
tendency of voters to identify closely with a political party, with a
corresponding reduction in the number of \textquotedblleft
uncommitted\textquotedblright\ or \textquotedblleft
issue-based\textquotedblright\ voters. At the ballot level this results in
more \textquotedblleft party-line\textquotedblright\ votes, with relatively
few \textquotedblleft split\textquotedblright\ or \textquotedblleft
cross-over\textquotedblright\ votes. While the outcome of an election depends
on who shows up to vote, and how they choose to vote; in a polarized society
the second factor is much less important. The political parties have a clear
idea of who their core supporters are, and much of the election effort is
focused on a get-out-the-vote campaign, targeted towards getting these
supporters show up on election day.

In this paper we study the likelihood of a \textquotedblleft single-party
sweep\textquotedblright; namely, given some set of elections, what is the
probability that \emph{a single party wins all of them?} In order for this to
be a well-defined question, one must clarify what it means to win an election.
In some real-world situations, such as a presidential contest or a two-party
parliamentary election, there is an unambiguous notion of what constitutes a
\textquotedblleft win\textquotedblright; in other settings, there might be
several natural candidates.

For example, in the case of Indian state elections, one might be interested in
analyzing either the total number of assembly seats won, or the total number
of state governments formed by a party. Alternatively one might choose to
focus on a single state, examining whether a single party wins majority in
that state as well as the national level. The ruling party in India refers to
this as ``double engine governance'', and portrays it as a favorable outcome
that accelerates developmental initiatives within the state\footnote{The Indian Railways 
occasionally operates trains with dual engines to enhance speed and reliability, particularly 
when handling heavier loads.}. For maximal
applicability of our analysis, we choose to work with a somewhat abstract
notion of a \emph{win}, subject only to two mild requirements, which fits all
these scenarios and many more besides. \smallskip

\emph{Our main finding is that in a country with a polarized electorate, a
more simultaneous polling schedule increases the likelihood of a single-party
sweep. } \smallskip

We make this statement precise below in three scenarios of increasing
generality. Our most general scenario allows for several parties and staggered
polling, which is often the case in Indian elections. Although our two
motivating examples, the US and India, both involve first-past-the-post
voting, \emph{this is not assumed in our model}, and our results are equally
applicable to several other major electoral systems, including party list
proportional representation.

Finally, it is worth noting that we actually prove a stronger result below;
namely, that probability of a sweep goes up for every party. Thus the effect
we describe is systemic in nature and does not depend on the coattails of a
popular leader. It stems solely from statistical aggregates of individual
voters' actions, and is best regarded as a macroscopic manifestion of
political polarization. Our model brings to light a novel cause for sweeps
which is not at odds with other causes discussed in the literature, but rather
independent of them.

\subsection{Main results}

Country A has two political parties and two national elections, Presidential
and Parliamentary; in each election the candidate with the most votes wins,
with ties decided by a coin toss. The election commission is considering
whether to hold the elections simultaneously or on separate polling dates.
Each voter $h$ has a fixed probability $p_{h}$ of voting on any poll
date\footnote{This {assumption is in the spirit} of Banzhaf \cite{Banzhaf} and
Penrose \cite{Penrose} ({see section \ref{Random Turnout}).}}, and the
population is \emph{polarized} in the sense that each voter prefers the same
party for both elections. A \emph{sweep} is an outcome where one party wins
both the presidency and a parliamentary majority.

\begin{alphatheorem}
\label{ThA} A simultaneous election makes a sweep more \footnote{Throughout,
{\textquotedblleft more\textquotedblright} is to be understood in the weak
sense of \textquotedblleft greater than or equal to\textquotedblright.} likely.
\end{alphatheorem}

Country B has several political parties, and several elections with different
notions of what constitutes a {\textquotedblleft win\textquotedblright}.
However the following two conditions hold.

(a) Each election can be ``won'' by \emph{at most} one party.

(b) For any turnout in any election, an extra vote for some party cannot
decrease its win probability\footnote{We speak of ``win probabilities'' since,
for example, tied vote counts might lead to coin tosses.}, nor can it increase
the win probability of a rival party.

The election commission is exploring various polling schedules. We will say
that schedule $\pi$ is \emph{coarser} than schedule $\pi^{\prime}$ if any two
elections that are simultaneous (have the same poll date) in $\pi^{\prime}$
are also simultaneous in $\pi$. As before, each voter $h$ has a fixed
probability $p_{h}$ of voting on any poll date and always votes for the same
party. A \emph{sweep} is an outcome in which one party wins all the elections.

\begin{alphatheorem}
\label{ThB} A coarser polling schedule makes a sweep more likely.
\end{alphatheorem}

Country C has a setup similar to country B. However the election commission is
considering \emph{staggered} schedules in which each election might be spread
out over several days, with different poll dates for different regions. We
will say that a staggered schedule $\Pi$ is coarser than $\Pi^{\prime}$ if, for every voter $h$, any
two elections that are simultaneous for $h$ in $\Pi^{\prime}$ are also
simultaneous for $h$ in $\Pi$.

\begin{alphatheorem}
\label{ThC} A coarser staggered schedule makes a sweep more likely.
\end{alphatheorem}

In a country with many political parties, it often happens that several
parties come together in an {alliance} in order to form a government. Here one
can distinguish two kinds of alliances -- a pre-poll alliance which involves
seat sharing, and a post-poll alliance in which a group of parties comes
together after having fought the elections separately, perhaps competing
against each other for more seats. In this case, one may ask whether Theorem
C continues to hold if we treat each {alliance} as single party. In section
\ref{alliances} we show that is indeed the case for many real-world electoral systems.

Regarding the proofs of these results, it is evident that Theorem \ref{ThC} implies  
Theorem \ref{ThB}, which in turn implies Theorem \ref{ThA}. We will deduce Theorem \ref{ThC} 
from Theorem \ref{ThD}, a mathematical result that we formulate and prove 
below. Theorem \ref{ThD} also yields a new proof of the Harris inequality -- 
a well-known correlation inequality with many mathematical
applications \cite{Alon-Spencer, Harris}. As our proof shows, Theorem
\ref{ThD} may be regarded as a generalization of the Harris inequality to $n$
functions.\footnote{The Harris inequality is a special case of the still more
general FKG inequality \cite{FKG}. One of the authors of the present paper has
proposed a different generalization of the FKG inequality to $n$ functions and
obtained some partial results in this direction. However the general problem
remains open, and seems much harder than the generalization considered here;
see \cite{LiebSahi, rich, sahi, sahi2, sahi3}.}

\subsection{Structure of the paper}

In section \ref{sec:C} we formulate and prove Theorem \ref{ThD}, and 
deduce Theorem \ref{ThC}. We also provide a new proof of the 
Harris inequality, formulated here as Theorem \ref{ThE}. 

In Section \ref{Random Turnout} we compare our model of random voter turnout to
other models of voter behavior. In Section \ref{electoral systems} we discuss
the applicability of our results to various electoral systems that are in use
around the world, and have been studied extensively in the theory of Social
Choice \cite{Arrow, Fishburn}. In Section \ref{alliances} we extend 
Theorem \ref{ThC} to party alliances, and in Section \ref{Sweeps} we contrast our 
anaysis of sweeps to the earlier literature. Finally, in Sections \ref{Coattail} and
\ref{ONOE} we discuss the implications of our results for elections in the USA and India, respectively.

\section{Proofs}

\label{sec:C} We will deduce Theorem \ref{ThC} from Theorem \ref{ThD}, which
we now formulate and prove.

Let $L=\{1,\ldots,n\}$ be a finite set, let $\pi$ be a partition of $L$, and
let $\mathcal{L}$ be the power set of $L$. Given $0\le p\le1$, construct a
random subset $S\in\mathcal{L}$ as follows: for each $M$ in $\pi$, toss a coin
with probability $p$ of heads; if heads then include $M$ in $S$.

\begin{definition}
We write $\mu_{\pi,p}$ for the probability measure on $\mathcal{L}$ such that
$\mu_{\pi,p}(S)$ is the probability of obtaining the subset $S$ by the above
random procedure.
\end{definition}

Let $H=\{1,\ldots,m\}$ be another finite set, and let $\mathcal{H}$ and
$\mathcal{T}$ be the power sets of $H$ and $H\times L$. Then we have
bijections $\lambda: \mathcal{L}^{m} \rightarrow\mathcal{T}, \; \theta:
\mathcal{H}^{n} \rightarrow\mathcal{T}$, given by
\[
\lambda(L_{1},\ldots,L_{m})=\bigcup\nolimits_{h} \{(h,l):l\in L_{h}\},
\quad\theta(H_{1},\ldots,H_{n}) =\bigcup\nolimits_{l} \{(h,l):h\in H_{l}\}.
\]

\begin{definition}
\label{def-align} We say that a (real-valued) function $f$ on $\mathcal{H}$ is
\emph{increasing} (resp.~\emph{decreasing}) at $h\in H$, if for all $S\subset
H$ one has $f(S \cup\{h\})- f(S)\ge 0$ \big(resp.~$\le 0$\big).

We say that  $F=(f_{1},\ldots, f_{n})$ is an \emph{aligned} tuple of functions 
on $\mathcal{H}$  if, for each $h\in H$, either all the $f_{l}$ are increasing at $h$,
or all are decreasing at $h$.
\end{definition}

Let $\Pi=(\pi_{1},\ldots, \pi_{m})$, $\mathbf{p}=(p_{1},\ldots, p_{m})$,
$F=(f_{1},\ldots, f_{n})$ be tuples such that each $\pi_{h}$ is a partition of
$L$, each $p_{h}$ is in $[0,1]$, and each $f_{l}$ is a function on
$\mathcal{H}$. We write $\mu_{\Pi, \mathbf{p}}$ and $\bar{F}$ for the product
measure and function on $\mathcal{T}$ induced by $\lambda$ and $\theta$:
\[
\mu_{\Pi, \mathbf{p}}(L_{1},\ldots, L_{m}):=\mu_{\pi_{1},p_{1}}(L_{1})
\cdots\mu_{\pi_{m},p_{m}}(L_{m}), \; \bar{F}(H_{1},\ldots, H_{n}):=f_{1}%
(H_{1})\cdots f_{n}(H_{n}).
\]

\begin{definition}
For $\Pi$, $\mathbf{p}$, $F$ as above, we define $E(\Pi,\mathbf{p}%
,F):=\sum\nolimits_{T\in\mathcal{T}} \bar{F}(T)\mu_{\Pi, \mathbf{p}}(T).$
\end{definition}

\begin{alphatheorem}
\label{ThD} Suppose $F=(f_{1},\ldots,f_{n})$ is an aligned tuple of
nonnegative functions on $\mathcal{H}$. If $\;\Pi^{\prime}$ is coarser than
$\Pi$ then we have $E(\Pi^{\prime},\mathbf{p},F)\ge E(\Pi,\mathbf{p},F)$ for
all $\mathbf{p}$.
\end{alphatheorem}

\begin{proof}
It suffices to prove $E(\Pi^{\prime},\mathbf{p},F)\ge E(\Pi,\mathbf{p},F)$ for
the case where

\begin{itemize}
\item[1)] $\Pi^{\prime}=\{\pi^{\prime}_{h}\}$ and $\Pi=\{\pi_{h}\}$ differ
only for a single $h$, say $h=1$;

\item[2)] $\pi^{\prime}_{1}$ is obtained by combining two parts of $\pi_{1}$,
say the first two, as follows:
\[
\pi_{1}=\{M_{1}, M_{2}, M_{3},\ldots, M_{k} \}, \quad\pi_{1}^{\prime}%
=\{M_{1}\cup M_{2}, M_{3},\ldots, M_{k} \}.
\]

\end{itemize}

Indeed, by iterating 2)~we obtain the inequality for any $\pi^{\prime}_{1}$
coarser than $\pi_{1}$. Now by iterating 1)~we obtain the inequality for any
$\Pi^{\prime}$ coarser that $\Pi$.

Thus we may assume that $\Pi$ and $\Pi^{\prime}$ are as in 1) and 2). Now the
measures $\mu_{\Pi, \mathbf{p}}$ and $\mu_{\Pi^{\prime}, \mathbf{p}}$ involve
the same coin tosses for $\pi_{2},\ldots,\pi_{m}$ and for $M_{3},\ldots,M_{k}$
in $\pi_{1}$. It suffices to show that for each subset $T^{\prime}$ of
$L\times H$ resulting from these tosses, the conditional expectation of $F(T)
= f_{1}(H_{1})\cdots f_{n}(H_{n})$ is higher if we toss a single coin for
$M_{1}\cup M_{2}$ rather than separate coins for $M_{1}$ and $M_{2}$.

To study this we write $\theta(T^{\prime})=(H_{1},\ldots,H_{n})$, and we set
\[
a_{i} = \prod_{l\in M_{i}} f_{l}(H_{l}\cup\{1\}), \; b_{i}= \prod_{l\in M_{i}}
f_{l}(H_{l}), \text{ for } i=1,2;\quad c= \prod_{l\notin M_{1}\cup M_{2}}
f_{l}(H_{l}).
\]
Then the conditional expectations for the single and double toss are,
respectively,
\[
A=[pa_{1}a_{2}+(1-p)b_{1}b_{2}]c, \; B= [p^{2}a_{1}a_{2}+ p(1-p)(a_{1}b_{2} +
b_{1}a_{2})+(1-p)^{2}b_{1}b_{2}]c,
\]
where $p=p_{1}$. Now by an elementary calculation we get
\[
A-B= p(1-p)(a_{1}-b_{1})(a_{2}-b_{2})c.
\]

Since $0\le p\le1$, we have $p(1-p)\ge0$, and since $f_{i}\ge0$, we have
$c\ge0$. Further, since the $f_{i}$ are aligned, either they are all
increasing or all decreasing, at $h=1$. In the first case we have $a_{i}\ge
b_{i}$ and in the second case $a_{i}\le b_{i}$, for $i=1,2$. In either case we
get $(a_{1}-b_{1})(a_{2}-b_{2})\ge0$, which implies $A-B\ge0$.
\end{proof}

\medskip

Theorem \ref{ThC} follows immediately from Theorem \ref{ThD}. \medskip

\begin{proof}
[Proof of Theorem \ref{ThC}]Let $H=\{1,\ldots,m\}$ be the set of voters and
let $L=\{1,\ldots,n\}$ be the set of elections. A voter \emph{turnout},
$T\subset H\times L,$ is the set of pairs $(h,l)$ such that voter $h$ has
voted in election $l$. The staggered schedules $\Pi$ and the voter
probabilities $\mathbf{p} = (p_{h})$ induce a probability measure on voter
turnouts, which is seen to be precisely $\mu_{\Pi, \mathbf{p}}$. We will prove
that if we replace $\Pi$ by a coarser partition then the probability of a
sweep goes up for \emph{every} party.

Fix a party $s$ and define $F=(f_{1},\ldots,f_{n})$ where $f_{i}(H^{\prime})$
is the probability that party $s$ wins election $i$ when $H^{\prime}$ is the
set of votes cast. Then the $f_{i}$ are non-negative and aligned by our
assumptions. Now the probability of a sweep by party $s$ is precisely $E(\Pi,
\mathbf{p},F)$, and so the result follows by Theorem \ref{ThD}.
\end{proof}

\medskip

Theorem \ref{ThD} also gives a new proof  of
the Harris correlation inequality \cite{Harris}. 
For $\mathbf{p}=(p_{1},\ldots, p_{m})$, define a probabilty measure 
and expectation on $\mathcal{H}$ as follows:
\[
\mu(S)=\mu_{\mathbf{p}}(S)=\prod\nolimits_{h\in S}(p_{h})\prod\nolimits_{h\notin S }(1-p_{h}),
\quad E(f,\mu)=\sum\nolimits_{S\in\mathcal{H}} f(S)\mu(S).
\]

\begin{alphatheorem}
\label{ThE} If $f_{1},f_{2}$ are increasing functions on $\mathcal{H}$ then we
have
\begin{equation}
\label{har}E(f_{1}f_{2},\mu) - E(f_{1},\mu)E(f_{2},\mu)\ge0.
\end{equation}

\end{alphatheorem}

\begin{proof}
This is due to Harris \cite{Harris}, but it is also an immediate corollary of
the case $n=2$ of Theorem \ref{ThD} as we now explain. First, if we add a
constant to $f_{1}$ or $f_{2}$ then the left side of \eqref{har} is unchanged,
thus we may assume without loss of generality that $f_{1}$ and $f_{2}$ are
non-negative. Evidently then the pair $F=(f_{1},f_{2})$ is aligned and
non-negative. Now let $\Pi=(\pi_{1},\ldots, \pi_{m})$ and $\Pi^{\prime}%
=(\pi^{\prime}_{1},\ldots, \pi^{\prime}_{m})$ where $\pi_{h}=\{\{1\},\{2\}\}$
and $\pi^{\prime}_{h}=\{\{1, 2\}\}$ for all $h$, then we have
\[
E(\Pi^{\prime},\mathbf{p},F)=E(f_{1}f_{2},\mu),\quad E(\Pi,\mathbf{p}%
,F)=E(f_{1},\mu)E(f_{2},\mu).
\]
Since $\Pi^{\prime}$ is coarser than $\Pi$, Theorem \ref{ThD} implies
\eqref{har} as desired.
\end{proof}

Thus Theorem \ref{ThD} can be regarded as an $n$-function generalization of teh Harris inequality.
\section{Discussion\label{Discussion}}

\subsection{Random Voter Turnout \label{Random Turnout}}

A key hypothesis of our model is that the probabilites of turning out to vote
are \emph{independent }across individuals, and also across different polling
dates for any given individual\footnote{Furthermore, while these probabilities
may vary arbitrarily across individuals, they remain constant (across
different polling dates) for each individual.}. This is very much in the
spirit of Penrose \cite{Penrose}, and later Banzhaf \cite{Banzhaf}), both of
whom considered the special case where these probabilities were identically
$1/2.$

The recent \textquotedblleft instrumental theory of turnout\textquotedblright%
\footnote{The theory posits individuals for whom the act of voting is
{instrumental in maximizing }their own \textquotedblleft
payoffs\textquotedblright, though these payoffs may be very broadly defined,
incorporating not only the joy of seeing their party win but also the joy of
participating in the election (alongside the cost of participating).}%
\ provides theoretical underpinnings for the independence assumption. Some of
the analyses in the field are focused on individual rational choice (in an
{exogenously specified }environment), modeled in terms of adaptive learning
(e.g., \cite{Kanazawa}), or utility maximization (e.g., \cite{Odershook}) or
minmax regret (e.g.,\cite{Ferejohn}). Others discuss full-blown strategic
interaction among the voters (and sometimes also candidates) and derive Nash
equilibria of \textquotedblleft participation games\textquotedblright\ (e.g.,
\cite{Ledyard}, \cite{Myerson}, \cite{Palfrey}). These analyses differ
considerably from one another on how individual turnout probabilities are
formed, but they all imply the independence of those probabilities ({see the
surveys in \cite{Dhillon} and }\cite{Geys}). Some recent models (e.g.,
\cite{Pogorelsky}), while maintaining the instrumentality assumption,
introduce the possibility of communication --- broadly defined --- between
candidates, media and voters. This gives rise to correlated equilibria where
the turnout probabilities of voters are no longer independent. (Also, there
are several behavioral models which openly depart from the \textquotedblleft
instrumental theory\textquotedblright\ and directly incorporate correlated
turnout among the voters, such as voting in \textquotedblleft
teams\textquotedblright\ (see \cite{Fedderson}, \cite{Nalebuff} and the
references therein).

The independence hypothesis is {central} to our analysis\footnote{We take it
to be a \emph{behavioral }hypothesis which governs voters' turnout. In other
words, no matter how the behavior might once have originated from rational
(optimal) choice, it has got entrenched as a \textquotedblleft rule of
thumb\textquotedblright\ for a voter (see \cite{Aumann}), which he follows
without bothering to recompute the exact optimal choice, each time the
election scenario changes on account of alterations in the polling schedule.}.
However there is a \textquotedblleft special kind\textquotedblright\ of
correlation that can be admitted in our model (which, while mathematically
{obvious, may be useful both conceptually and for applications}). {As in
Harsanyi (\cite{Harsanyi}), }define the \emph{type} of voter $h\in H$ to
consist of two components: the party $t\in T$ that $h$ will vote for; and the
probability $p_{h}$ with which $h$ will {turn out to vote on \emph{any}
polling date}. {Before the announcement of the polling schedule, assume that
nature picks the vector of voter-types according to some \textit{a priori}
probability distribution on a given set.\footnote{{A \textquotedblleft
move\textquotedblright\ of nature could represent public news that affects
voters' types; or --- alternatively --- independendent (idiosyncratic)
perturbations in every voter's }$p_{h}${.}} }Thus nature's move is \emph{ex
ante}, and every voter's type stays fixed throughout the elections. Now, if
schedule $\Pi$ is coarser than $\Pi^{\prime}$, then the probability of a sweep
is higher under $\Pi$ than under $\Pi^{\prime}$ contingent on \emph{every
}move of nature (by Theorem \ref{ThC}), therefore obviously also in
expectation across all the moves of nature. (However, if nature were to move
\emph{ex post }independently between different poll dates, the argument given
above breaks down.)

\subsection{Electoral systems \label{electoral systems}}

There is a wide variety of election methods in use across the world. For an
extensive survey, covering all the methods we cite below, see \cite{ACE}.

Our analysis applies to first-past-the-post elections that are held in several
countries including India, USA, Canada and UK (the first two of which we shall
discuss below in some detail).

It also applies to elections based on proportional representation, where each
person votes for a single party of his choice, and the percentage of
parliamentary seats accorded to any party equals (as nearly as possible) the
percentage of total votes it received. Of course, for our analysis to be
valid, the method for \textquotedblleft rounding\textquotedblright\ the seats
(into integers) should be \textquotedblleft monotonic\textquotedblright%
\ \footnote{Monotonicity holds for most of the rounding methods used in
practice, e.g., the Highest Average or the Largest Remainder Methods
\cite{ACE}} in the obvious sense, so that condition (b) is not violated; and,
furthermore, condition (a) must also hold (e.g., the party with the most seats
is deemed to win; {or, alternatively, if no party gets a strict majority of
the seats, it is a non-win for every party}). Such elections are in use ---
with minor modifications --- in many countries and also in the European
Parliament (see, again, \cite{ACE} for details).

There is, however, a category of elections that lies outside the ambit of our
model, and it includes the Two Round System (used in the French Presidential
Election), Single Transferable Vote, Alternative Vote, Supplementary Vote,
Borda Count and so on (described in detail in \cite{ACE}). In all these cases,
{a vote is tantamount to a ranking }of the {electoral candidates}. Such
elections have received the lion's share of attention in the theoretical
literature on \textquotedblleft social choice\textquotedblright\ (SC)\ ever
since Arrow's pioneering monograph \cite{Arrow}. (For a succinct survey, see,
e.g., \cite{Fishburn}.) The concerns of the SC literature are quite different
from ours, but even from a \textquotedblleft high level\textquotedblright,
there are salient differences. First, our voter only specifies his top choice
instead of ranking the candidates. More importantly, in SC, the entire voter
population is {assumed to always} {be present}. In sharp contrast, it is the
\emph{variable, stochastic turnout} of voters on which{ our analysis turns}.
It might be interesting to examine, along the lines of this paper, how the
probability distribution of electoral outcomes depends on the stochastic
turnout, in the context of the more sophisticated elections considered in SC
(e.g., Borda Count, see \cite{Maskin}), but that would be a topic for another paper.

\subsection{Party Alliances \label{alliances}}

Consider a partition of parties into sets, each of which constitutes an
electoral alliance. Define the \emph{type }of an alliance to be
\emph{pre-poll} if the parties in the alliance jointly field one candidate in
every election, and do not contend against one another; and to be
\emph{post-poll} if they contend every election on their own, but come
together afterwards by pooling their elected candidates. Our analysis will
remain intact if conditions (a) and (b) hold, viewing each alliance as a
single contender in all the elections. This is easily seen to be so in
elections based on first-past-the post criterion or on proportional
representation, provided the type of each alliance is fixed across all polling
schedules (i.e., either it is pre-poll for all schedules, or post-poll for
all). To check (b) in the presence of post-poll alliances (the only scenario
where it is not immediately obvious), note that if a vote is cast in favor of
an alliance it does not benefit any of the \emph{rival} alliances, and
therefore cannot hurt the alliance either, since the total number of elected
candidates is fixed in every election.

However, if the set of parties that form an alliance, or the type of an
alliance, varies across elections, or across the polling schedules, or with
the electoral turnout, then our analysis breaks down.

\subsection{Single-Party Sweeps \label{Sweeps}}

A single-party sweep (also known as \textquotedblleft one-party
monopoly\textquotedblright\ or \textquotedblleft one-party
dominance\textquotedblright) is a widely prevalent phenomenon, both in local
and national elections (\cite{Parry}, \cite{Pempel}, \cite{Solinger}) Its
causes have been discussed from different standpoints, including: the
strategic timing of elections ( \cite{Kayser}), economic intervention
(\cite{Kayser}), manipulation of electoral rules (\cite{McElwain} ), and
fragmented opposition (\cite{Ziegfield}). There does not seem to be much in
the literature connecting polling schedules to single-party sweeps, which is
the focus of this paper. Two notable exceptions are the \textquotedblleft
coattail effect\textquotedblright\ in the context of US elections (see, e.g.,
\cite{Calvert, Campbell, Meredith}). and a recent study
(\cite{Balasubramaniam}) of the Indian elections. {Both analyses are based on}
the postulate that a voter's behavior \emph{changes }when elections are held
simultaneously, on account of either the popularity of the Presidential
candidate or the \textquotedblleft salience\textquotedblright\ of one of the
political parties {(see the next two sections)}. In sharp contrast, voters'
behavior is \emph{fixed} in our model, regardless of the polling schedule or
the characteristics of the candidates standing for election. Nevertheless the
voter turnout\textbf{ }varies with the schedule and drives our result.

\subsection{US Elections and The Coattail Effect \label{Coattail}}

The data from USA shows that a single party is more likely to win multiple
elections when they are held simultaneously.{ Indeed, during Congressional
elections in USA that coincide with the presidential race, the party of a}
popular president tends to win more seats in Congress compared to elections
held in the midterm of the Presidency (the latest instance of this being the
US elections in November 2024). This phenomenon is referred to as the
\textquotedblleft coattail\textquotedblright\ effect. {The rationale is that a
popular }presidential candidate fosters additional wins for his party in
Congress, with many members of Congress voted into office \textquotedblleft on
the coattails\textquotedblright\ of the president.

Our analysis brings to light a completely different mechanism for a
single-party sweep, which is not at odds with the coattail effect, but
independent of it. Two behavioral causes are commonly postulated for the
coattail effect: a popular presidential candidate can mobilize each member in
his party base to turn out to vote with higher probability; and, moreover, he
can sway undecided voters to his party's fold.\footnote{A cause of a different
nature is the \textquotedblleft anti-incumbency\ factor\textquotedblright\ so
often alluded to by political analysts. It can come into play during midterm
elections, especially if there is widespread perception that the ruling party
has been unable to fulfil its electoral promises. For then the rival party
makes gains in the midterm elections, further accentuating the benefit of
simultaneous elections for the ruling party.} These causes are ruled out in
our model since the party that $h$ will vote for, as well the\emph{
}probability $p_{h}$ that he will vote on any poll date, are both exogenously
fixed for every $h\in H$ and not susceptible to change. It is purely
\textquotedblleft the statistics of voting\textquotedblright\ {that enhances
the single-party sweep for simultaneous elections}.\footnote{{We could
incorporate the \textquotedblleft coattail\textquotedblright\ in our model as
an increase in the probabilities of turnout of the adherents of the
President's party in (at least) all other contemporaneous Congressional
elections (while the turnout for its rivals stay fixed or go down). This will
obviously increase the probability of win for the party in every election. In
this sense, our model is accomodative of the coattail effect.}}

The probability of a single-party sweep can, of course, be very small if there
are many parties or many elections. It becomes significant when, as is often
the case, there are two parties contending two (or three) elections. To the
extent that \textquotedblleft Duverger's Law\textquotedblright\ operates and a
representative democracy settles into a two-party system, the sweep can be
quite a robust phenomenon \cite{Duverger, Fey, Riker}

\subsection{Indian Elections: \textquotedblleft One Nation, One Election\textquotedblright\ and
\textquotedblleft Double Engine  Governance\textquotedblright\ \label{ONOE}}

Elections in India are held as follows. Political parties contend for seats in
the Lok Sabha at the national level and also in $28$ Vidhan Sabhas, one for
each state. The country (resp., state) is partitioned into electoral districts
for the Lok (resp., Vidhan) Sabha election, the candidate with the most votes
is elected from each district, and the party with the most candidates is
deemed to have won the election.\footnote{W.l.o.g. we may assume that the
entire electorate of India participates in each Vidhan Sabha election, viewing
voters outside the state as \textquotedblleft strategic
dummies\textquotedblright\ whose votes do not count {(since condition (b)
still holds for such dummies.)}}

The possibility of simultaneous elections is no longer theoretical but it
looms large in India and is the subject of a fierce ongoing debate, set into
motion by the BJP\ and its allies (currently in power in the Lok Sabha)
seeking a Constitutional Amendment to implement their \textquotedblleft One
Nation, One Election\textquotedblright\ (ONOE)\ proposal (which would make all
Vidhan Sabha elections simultaneous with the Lok Sabha election). At the same
time, the BJP has been strongly campaigning for \textquotedblleft
double-engine\textquotedblright\ governance, especially during elections in
the states. Theorem \ref{ThC} shows that the likelihood of one party winning
the Lok Sabha and the Vidhan Sabha of any given state is higher when both
elections are simultaneous compared to when they are not, i.e., under ONOE,
the probability of \textquotedblleft double-engine\textquotedblright%
\ governance goes up for every state. {This happens by way of a }%
\emph{systemic change}{, no matter which party wins the Lok Sabha.} {The
magnitude of the change can, in general, be quite significant\footnote{{We
thank Vijay Vazirani for asking this question.}}}. Indeed consider two
parties, BJP and its allies versus the \textquotedblleft
INDIA\textquotedblright\ coalition of all their rivals (as in the 2024
elections) contending the Lok Sabha at the Center and the Vidhan Sabha in
Bihar. For ease of computation assume that each contender is favored by half
the population of India, and of Bihar; and that the turnout probability is
$1/2$ {for every voter}$.$ {Then, by the law of large numbers, the probability
of a sweep ( i.e., of double-engine governance) rises from about }$50\%$ {to
about }$100\%$ {when the polling schedule of the two elections is made
simultaneous (and rises from }$25\%$ {to }$50\%$ {for each contender). }

The study in \cite{Balasubramaniam} refers to the pronounced correlation in
the Indian data between simultaneous elections and single-party sweeps (over
those elections). To explain it, the authors posit a model of
behaviorally-constrained voters for whom information acquisition is costly. If
a party is salient in the Lok Sabha elections, then the electorate has a
higher probability to also vote for it in any contemporaneous Vidhan Sabha
election, rather than acquire costly information about its rivals which could
alter that decision. Such is not the case if the Vidhan Sabha election is held
at a very different time, for then the \textquotedblleft salience
factor\textquotedblright\ is missing and information acquisition becomes less
costly. By fixing the type of a player during elections, we have ruled out
this kind of varition in voter behavior. Our model highlights an alternative,
purely stastistical cause of the correlation.

Finally recall that the notion of \textquotedblleft win\textquotedblright\ is
very general in our model and can vary across elections, subject only to
conditions (a) and (b). As mentioned in the introduction\textbf{, }this
enables us to consider Indian elections from diverse viewpoints, including the
following two starkly different ones: $1+28=29$ \textquotedblleft
literal\textquotedblright\ elections, or $2$ \textquotedblleft
figurative\textquotedblright\ elections (with the $28$ Vidhan Sabha elections
viewed as one conglomorate). The literal elections clearly fit into our model,
but here the likelihood of a sweep is so miniscule that it is unlikely for
either party to be focused on it. Sweeps occur with higher ptobability in the
$2$-election framework, where \textquotedblleft win\textquotedblright\ in
Vidhan Sabhas can be defined in diverse ways, e.g., (i) winning a specific set
of states\footnote{{In particular, the set could consist of a single state, in
which case a sweep reduces to the \textquotedblleft double-engine
governance\textquotedblright\ discussed earlier. More generally, in theory if
not in practice, the contending parties may have different sets of states in
their sights but, so long as the intersection of their sets is non-empty,
condition (a) holds and our analysis remains intact. }}; (ii) winning most
states overall; (iii) winning most seats overall (relevant for the future
Rajya Sabha); and so on. All these scenarios are covered by Theorem \ref{ThC},
which allows for staggered polling schedules (i.e., personal partitions of the
elections in $L)$.

\end{document}